\documentclass[11pt,a4paper]{amsart}

\usepackage[margin=1.25in,a4paper]{geometry}
\usepackage{lmodern}
\usepackage[utf8]{inputenc}
\usepackage[T1]{fontenc}

\usepackage{microtype}
\usepackage[shortlabels]{enumitem}

\usepackage[english]{babel}
\usepackage{mathrsfs}

\usepackage{xcolor}
\definecolor{linkcol}{RGB}{0,80,158}
\definecolor{citecol}{RGB}{46,117,120}

\usepackage[colorlinks]{hyperref}
\hypersetup{
    colorlinks=true,
    linkcolor=linkcol,
    urlcolor=linkcol,
    citecolor=citecol,
}
\usepackage[hyperpageref]{backref}
\usepackage[msc-links,alphabetic]{amsrefs}
\usepackage[capitalize,nameinlink]{cleveref}
\crefname{equation}{}{}

\theoremstyle{plain}
\newtheorem{theorem}{Theorem}[section]
\newtheorem{lemma}[theorem]{Lemma}
\newtheorem{corollary}[theorem]{Corollary}

\theoremstyle{definition}
\newtheorem*{definition}{Definition}
\newtheorem{example}[theorem]{Example}
\theoremstyle{remark}
\newtheorem*{remark}{Remark}

\usepackage{thm-restate}

\newcommand{\bC}{\mathbb{C}}
\newcommand{\bD}{\mathbb{D}}
\newcommand{\bQ}{\mathbb{Q}}
\newcommand{\bR}{\mathbb{R}}
\newcommand{\bZ}{\mathbb{Z}}
\newcommand{\cF}{\mathcal{F}}
\newcommand{\sA}{\mathscr{A}}
\newcommand{\sD}{\mathscr{D}}
\newcommand{\sH}{\mathscr{H}}
\newcommand{\sP}{\mathscr{P}}
\newcommand{\ap}{\mathrm{ap}}
\newcommand{\ext}{\mathrm{ext}}
\newcommand{\abs}[1]{\left\lvert#1\right\rvert}
\newcommand{\norm}[1]{\left\lVert#1\right\rVert}
\renewcommand{\Re}{\operatorname{Re}}
\renewcommand{\Im}{\operatorname{Im}}
\newcommand{\dd}{\mathrm{d}}

\title[Analytic almost periodic functions]{Montel's theorem and composition operators for analytic almost periodic functions}
\author{Viktor Andersson}
\address{Department of Mathematical Sciences, Norwegian University of Science and Technology (NTNU), 7491 Trondheim, Norway}
\email{viktor.andersson@ntnu.no}
\date{\today}

\begin{document}
\begin{abstract}
    We consider the Banach space $H^\infty_\ap(\bC_0)$ of bounded analytic functions on the open right half-plane $\bC_0$ that are almost periodic on some smaller half-plane, as well as the subspace $A_\ap(\bC_0)$ of those functions in $H^\infty_\ap(\bC_0)$ that are uniformly continuous on $\bC_0$. We prove a strong version of Montel's theorem for $H^\infty_\ap(\bC_0)$ and characterize the bounded composition operators on $H^\infty_\ap(\bC_0)$ and $A_\ap(\bC_0)$, as well as the compact composition operators on $H^\infty_\ap(\bC_0)$ and certain subspaces of it.
\end{abstract}
\thanks{The author is supported by Grant 354537 of the Research Council of Norway.}
\maketitle
\section{Introduction}

A complex-valued function $f$ defined on a half-plane $\bC_\kappa=\{s\in\bC:\Re s>\kappa\}$ is called \emph{almost periodic} if it is bounded, uniformly continuous, and the set $E_f(\varepsilon)$, consisting of all $\tau\in\bR$ such that
$$\abs{f(s+i\tau)-f(s)}\leq\varepsilon$$
for all $s\in\bC_\kappa$, is relatively dense for all $\varepsilon>0$. Here, a set $A\subseteq\bR$ is called \emph{relatively dense} if there exists a $d>0$ such that any closed interval of length $d$ intersects $A$. In this paper, we will be concerned with bounded analytic almost periodic functions. Denote by $H^\infty(\bC_0)$ the Banach space of bounded analytic functions in the right half-plane $\bC_0$ equipped with the supremum norm. We will be interested in the following subspace of $H^\infty(\bC_0)$.

\begin{definition}
    Let $H^\infty_\ap(\bC_0)$ denote the space of all $f\in H^\infty(\bC_0)$ that are almost periodic on $\bC_\kappa$ for some $\kappa>0$.
\end{definition}

The main motivation for studying $H^\infty_\ap(\bC_0)$ is the space $\sH^\infty$, consisting of all $f\in H^\infty(\bC_0)$ that can be written as a convergent Dirichlet series
$$f(s)=\sum_{n=1}^\infty a_nn^{-s}$$
on some half-plane $\bC_\kappa$ with $\kappa>0$. A celebrated theorem of H. Bohr \cite{bohr_uber_1913} asserts that the Dirichlet series of any $f\in\sH^\infty$ converges uniformly on $\bC_\kappa$ for \emph{all} $\kappa>0$. From this we observe the inclusions
\begin{equation}\label{eq:subspaces}
    \sH^\infty\subseteq H^\infty_\ap(\bC_0)\subseteq H^\infty(\bC_0).
\end{equation}
Using standard techniques one can show that if $f\in H^\infty_\ap(\bC_0)$, then $f$ is almost periodic on $\bC_\kappa$ for \emph{all} $\kappa>0$, yielding an analogue to Bohr's theorem (see \cref{cor:bohr-analogue} below). From this it also follows that $H^\infty_\ap(\bC_0)$ is a closed subspace of $H^\infty(\bC_0)$, and since $\sH^\infty$ is complete (e.g. \cite{queffelec_diophantine_2020}*{Theorem 6.1.2}), we see that the inclusions in \cref{eq:subspaces} hold as closed subspaces. The goal of this paper is to show that several results that hold for $\sH^\infty$ have analogous results in $H^\infty_\ap(\bC_0)$.

The first result of this character is a strong version of Montel's theorem. In \cite{bayart_hardy_2002}, F. Bayart showed that any uniformly bounded sequence in $\sH^\infty$ has a subsequence that converges uniformly on all half-planes $\bC_\kappa$ with $\kappa>0$ to some function in $\sH^\infty$, strengthening the classical Montel theorem in this setting. This result does not hold in $H^\infty_\ap(\bC_0)$, and there are many counterexamples one can easily construct (see \cref{ex:no-uniformly-convergent-subsequence} below). Perhaps the simplest one is the sequence $\{f_n\}_{n\in\bZ^+}$ defined by $f_n(s)=e^{-s/n}$, which is clearly uniformly bounded. If this strong version of Montel's theorem was true for $H^\infty_\ap(\bC_0)$, then this sequence would have a subsequence converging uniformly on all half-planes $\bC_\kappa$ with $\kappa>0$. Since the sequence converges pointwise to $1$, and
$$\sup_{s\in\bC_\kappa}\lvert e^{-s/n}-1\rvert\geq\lim_{\sigma\to\infty}(1-e^{-\sigma/n})=1$$
for all $\kappa>0$, this cannot happen. We will prove a characterization of precisely for which uniformly bounded families in $H^\infty_\ap(\bC_0)$ the strong version of Montel's theorem holds. To state our theorem, we introduce a notion which is inspired by a similar definition\footnote{We remark that Bohr used the term ``uniformly almost periodic'' to describe such a family. We avoid this as the term ``uniformly almost periodic'' is commonly also used for what we call almost periodic functions (e.g. \cite{besicovitch_almost_1955}).} by Bohr in \cite{bohr_contribution_1943} for almost periodic functions on the real line.

\begin{definition}
    A family $\cF$ of almost periodic functions on a half-plane $\bC_\kappa$ is called \emph{jointly almost periodic} if the set $E_\cF(\varepsilon)=\bigcap_{f\in\cF}E_f(\varepsilon)$ is relatively dense for all $\varepsilon>0$.
\end{definition}

Our characterization is then as follows.

\begin{restatable}{theorem}{strongmontel}\label{thm:strong-montel}
    Let $\cF$ be a uniformly bounded family in $H^\infty_\ap(\bC_0)$. The following are equivalent:
    \begin{enumerate}[(i)]
        \item\label{item:jointly-ap-somewhere} $\cF$ is jointly almost periodic on some half-plane $\bC_\kappa$ with $\kappa>0$.
        \item\label{item:jointly-ap-everywhere} $\cF$ is jointly almost periodic on all half-planes $\bC_\kappa$ with $\kappa>0$.
        \item\label{item:strong-montel} For every sequence $\{f_n\}_{n\in\bZ^+}$ in $\cF$ there exists a function $f\in H^\infty_\ap(\bC_0)$ and a subsequence $\{f_{n_k}\}_{k\in\bZ^+}$ that converges uniformly to $f$ on $\bC_\kappa$ for all $\kappa>0$.
    \end{enumerate}
\end{restatable}

The equivalence of \labelcref{item:jointly-ap-somewhere} and \labelcref{item:jointly-ap-everywhere} can be shown by standard techniques of analytic functions on half-planes, and so arguably the more interesting equivalence is that of \labelcref{item:jointly-ap-everywhere} and \labelcref{item:strong-montel}. The main idea behind showing that \labelcref{item:jointly-ap-everywhere} implies \labelcref{item:strong-montel} is to first apply the classical Montel theorem, and then use joint almost periodicity to first extend the uniform convergence to vertical lines, and then to half-planes by analyticity. For the reverse implication, the key idea is to realize that condition \labelcref{item:strong-montel} is equivalent to having total boundedness of the family on each half-plane in the uniform norm. Recall that a metric space $(X,d)$ is called \emph{totally bounded} if for all $\varepsilon>0$ there exist some $x_1,\dots,x_n\in X$ such that if $x\in X$, then $d(x,x_k)<\varepsilon$ for some $1\leq k\leq n$. Using total boundedness, the problem of showing joint almost periodicity then reduces to the fact that any finite family of almost periodic functions is jointly almost periodic, which can be shown by techniques dating back to Bohr \cite{bohr_zur_1925} (see \cref{thm:finite-joint-ap} below).

Our approach highlights the importance of almost periodicity, and avoids the need for Bohr's theorem and Dirichlet series techniques, which are central tools in the proof of Bayart's theorem. Using our result, we also give an alternative proof of a recent extension of Bayart's theorem to significantly larger class of subspaces of $H^\infty_\ap(\bC_0)$ by Defant, Vidal, Schoolmann, and Sevilla-Peris (see \cite{defant_frechet_2021}*{Theorem 3.4} and \cref{thm:defant-montel} below).

Our second consideration concerns an analogue to the disk algebra in the setting of bounded analytic almost periodic functions. In \cite{aron_dirichlet_2017}, Aron, Bayart, Gauthier, Maestre, and Nestoridis introduced the space $\mathscr A(\bC_0)$ of all $f\in\sH^\infty$ that are uniformly continuous on $\bC_0$, and showed that this space consists of precisely the uniform limits of Dirichlet polynomials in $\bC_0$, i.e., functions $f$ on $\bC_0$ of the form
$$f(s)=\sum_{n=1}^N a_nn^{-s}$$
with $a_1,\dots,a_N\in\bC$. We define a corresponding space $A_\ap(\bC_0)$ for $H^\infty_\ap(\bC_0)$.

\begin{definition}
    Let $A_\ap(\bC_0)$ denote the space of all $f\in H^\infty_\ap(\bC_0)$ that are uniformly continuous on $\bC_0$.
\end{definition}

Recall that a general Dirichlet polynomial is a function $f$ on $\bC_0$ of the form
$$f(s)=\sum_{n=1}^Na_ne^{-\lambda_ns}$$
with $a_1,\dots,a_N\in\bC$ and $0=\lambda_1<\dots<\lambda_N$. We shall write $\sP_D$ for the space of all general Dirichlet polynomials. Our main result on $A_\ap(\bC_0)$ is then as follows.

\begin{restatable}{theorem}{apalgebradirichlet}\label{thm:ap-algebra-dirichlet}
    It holds that
    $$A_\ap(\bC_0)=\overline{\sP_D},$$
    where the closure in taken in the uniform norm.
\end{restatable}

We remark that, unlike the proof of Aron, Bayart, Gauthier, Maestre, and Nestoridis for the analogous result on $\sA(\bC_0)$, our proof does not rely on Bohr's theorem, but instead only on standard approximation results in the theory of almost periodic functions. This approach can also be used to give an alternative proof of their result.

Our final consideration concerns composition operators on $H^\infty_\ap(\bC_0)$, as well as on subspaces of it. In \cite{bayart_hardy_2002}, Bayart completely characterized the symbols $\varphi:\bC_0\to\bC_0$ that induce bounded composition operators $C_\varphi:\sH^\infty\to\sH^\infty$. Bayart showed that these are precisely the functions of the form $\varphi(s)=c_0s+\psi(s)$ for some non-negative integer $c_0$ and some analytic $\psi$ that can be represented as a Dirichlet series that converges uniformly on $\bC_\kappa$ for all $\kappa>0$ (see also \cite{bayart_composition_2021}*{Theorem 2.5}). We prove the following result for $H^\infty_\ap(\bC_0)$.

\begin{restatable}{theorem}{boundedcomposition}\label{thm:bounded-composition-characterization}
    Let $\varphi:\bC_0\to\bC_0$ be a function. The following are equivalent:
    \begin{enumerate}[(i)]
        \item\label{item:bounded-composition} $\varphi$ induces a bounded composition operator $C_\varphi:H^\infty_\ap(\bC_0)\to H^\infty_\ap(\bC_0)$.
        \item\label{item:composition-is-in-hinftyap} $f\circ\varphi\in H^\infty_\ap(\bC_0)$ for all $f\in H^\infty_\ap(\bC_0)$.
        \item\label{item:exponentials-are-in-hinftyap} $e^{-\lambda\varphi}\in H^\infty_\ap(\bC_0)$ for all $\lambda>0$.
        \item\label{item:exponential-and-analytic} $e^{-\lambda\varphi}\in H^\infty_\ap(\bC_0)$ for some $\lambda>0$, and $\varphi$ is analytic.
        \item\label{item:phi-of-form} $\varphi$ is of the form $\varphi(s)=as+\psi(s)$ for some $a\geq0$ and some analytic function $\psi$ that is almost periodic on $\bC_\kappa$ for all $\kappa>0$.
    \end{enumerate}
\end{restatable}

In \cite{contreras_composition_2024}, Contreras, Gómez-Cabello, and Rodríguez-Piazza characterized the symbols $\varphi:\bC_0\to\bC_0$ inducing bounded composition operators $C_\varphi:\sA(\bC_0)\to\sA(\bC_0)$. We prove an analogous result characterizing the bounded composition operators on $A_\ap(\bC_0)$ (see \cref{thm:bounded-composition-characterization-algebra} below).

We also investigate the matter of compactness. For $\sH^\infty$, this was done in \cite{bayart_hardy_2002}, where it was shown that if $\varphi$ induces a bounded composition operator $C_\varphi:\sH^\infty\to\sH^\infty$, then $C_\varphi$ is compact if and only if $\varphi(\bC_0)\subseteq\bC_\kappa$ for some $\kappa>0$. For $H^\infty_\ap(\bC_0)$, the characterization is similar but not the same. In particular, we show the following.

\begin{restatable}{theorem}{compactcomposition}\label{thm:compact-composition}
    Let $\varphi:\bC_0\to\bC_0$ induce a bounded composition operator $C_\varphi:H^\infty_\ap(\bC_0)\to H^\infty_\ap(\bC_0)$. The following are equivalent:
    \begin{enumerate}[(i)]
        \item\label{item:compact} $C_\varphi$ is compact.
        \item\label{item:compactly-contained} $\varphi(\bC_0)$ is compactly contained in $\bC_0$.
    \end{enumerate}
\end{restatable}

For subspaces of $H^\infty_\ap(\bC_0)$ where uniform boundedness implies joint almost periodicity on some half-plane---such as $\sH^\infty$---the result is in direct analogue with Bayart's result for $\sH^\infty$. If $W$ is a subspace of $H^\infty_\ap(\bC_0)$ with the property that any uniformly bounded family in $W$ is jointly almost periodic on $\bC_\kappa$ for some $\kappa>0$, and $\varphi$ is a symbol inducing a bounded composition operator $C_\varphi:W\to H^\infty_\ap(\bC_0)$, then we show that the assumption that $\varphi(\bC_0)\subseteq\bC_\kappa$ for some $\kappa>0$ is sufficient for $C_\varphi$ to be compact (see \cref{thm:compact-composition-w-sufficiency} below), and if $W$ is large enough in a certain sense, then this condition is also necessary (see \cref{thm:compact-composition-w-necessity} below). The proof of this is a direct application of \cref{thm:strong-montel}, analogous to how Bayart used his Montel theorem in characterizing the compact composition operators on $\sH^\infty$.

\subsection*{Acknowledgments}
I would like thank my PhD supervisor, Ole Fredrik Brevig, for introducing me to the topic of analytic almost periodic functions, and for his constant feedback in the writing of this paper.

\subsection*{Organization}
The remainder of this paper is divided into three sections. \Cref{sec:joint-ap-and-montel} starts by discussing some basic results on jointly almost periodic families, then moves on to the proof of \cref{thm:strong-montel}, and finishes with a proof of how the Montel theorem of Defant, Vidal, Schoolmann, and Sevilla-Peris follows from \cref{thm:strong-montel}. In \cref{sec:ap-algebra}, the basic properties of $A_\ap(\bC_0)$ are discussed, culminating in a proof of \cref{thm:ap-algebra-dirichlet}. The final section, \cref{sec:composition-operators}, deals with boundedness and compactness of composition operators on $H^\infty_\ap(\bC_0)$ and certain subspaces of it, and includes the proofs of \cref{thm:bounded-composition-characterization,thm:compact-composition}.

\section{Joint almost periodicity and a strong Montel theorem}\label{sec:joint-ap-and-montel}

We start by introducing some notation. For $\tau\in\bR$, we shall write $V_\tau f$ for the \emph{vertical translation}
$$V_\tau f(s)=f(s+i\tau).$$
If $f$ is almost periodic, then we denote by $d_f(\varepsilon)$ the minimal $d\geq0$ such that $E_f(\varepsilon)$ intersects any closed interval of length $d$. Similarly, for a jointly almost periodic family $\cF$, we denote by $d_\cF(\varepsilon)$ the minimal $d\geq0$ such that $E_\cF(\varepsilon)$ intersects any closed interval of length $d$. An element of $E_f(\varepsilon)$ will be called an \emph{$\varepsilon$-translation number of $f$}, and an element of $E_\cF(\varepsilon)$ will be called a \emph{common $\varepsilon$-translation number of $\cF$}.

Crucial to several of the arguments in this text is the fact that any finite family of almost periodic functions is jointly almost periodic. This is not as trivial as one may assume, as in general the intersection of two relatively dense set is not necessarily relatively dense; for example $\bQ$ and $\bR\setminus\bQ$ are both relatively dense but have empty intersection. To prove this, one can adapt the argument in \cite{besicovitch_almost_1955}*{Chapter 1, §1, Theorem 11}.

\begin{theorem}\label{thm:finite-joint-ap}
    Any finite collection of almost periodic functions on a half-plane $\bC_\kappa$ is jointly almost periodic.
\end{theorem}

As was mentioned in the introduction, it is the case that any function in $H^\infty_\ap(\bC_0)$ is almost periodic on $\bC_\kappa$ for all $\kappa>0$. We even have the following stronger result on extending the joint almost periodicity of a family from one half-plane $\bC_\kappa$ with $\kappa>0$ to all half-planes $\bC_\kappa$ with $\kappa>0$.

\begin{theorem}\label{thm:extend-joint-ap}
    Let $\cF$ be a family in $H^\infty(\bC_0)$ that is jointly almost periodic on some half-plane $\bC_\kappa$ with $\kappa>0$. Then $\cF$ is jointly almost periodic on all half-planes $\bC_\kappa$ with $\kappa>0$.
\end{theorem}

The result is a consequence of the Hadamard three-lines theorem, and the proof is a simple adaptation Besicovitch's proof of \cite{besicovitch_almost_1955}*{Chapter 3, §2, Theorem 4}. We leave the details to the reader.

\begin{remark}
    Using boundedness on the entire right half-plane, one can easily weaken the assumption in \cref{thm:extend-joint-ap} to only assuming joint almost periodicity on a single vertical line $\Re s=\sigma$.
\end{remark}

Taking now the family in the theorem to only consist of one function, we obtain our analogue result to Bohr's theorem, which was mentioned in the introduction.

\begin{corollary}\label{cor:bohr-analogue}
    If $f\in H^\infty_\ap(\bC_0)$, then $f$ is almost periodic on $\bC_\kappa$ for all $\kappa>0$.
\end{corollary}

In \cite{brevig_extension_2024}, Brevig and Kouroupis extended Bohr's theorem by showing that if $f$ is an analytic function on some half-plane $\bC_\kappa$ which can be represented as a convergent Dirichlet series on some half-plane $\bC_\theta$ with $\theta>\kappa$, and $f(\bC_\kappa)\subseteq\bC\setminus\{\alpha,\beta\}$ for some complex numbers $\alpha\neq\beta$, then the Dirichlet series of $f$ converges uniformly on $\bC_\nu$ for all $\nu>\kappa$. They prove this by showing that such an $f$ must be bounded on all half-planes $\bC_\nu$ with $\nu>\kappa$, and then use Bohr's theorem to arrive at the conclusion. Using \cref{cor:bohr-analogue}, we can prove an analogous result using the same argument. The main tool for the proof is the following quantitative version of Schottky's theorem due to Ahlfors; see \cite{ahlfors_extension_1938}*{Theorem B} for a proof.

\begin{theorem}[Schottky]\label{thm:schottky}
    Let $f$ be an analytic function defined on the open ball $B_r(c)$ in $\bC$, and suppose that $f(B_r(c))\subseteq\bC\setminus\{0,1\}$. Then
    $$\log\abs{f(s)}\leq\frac{r+\abs{s-c}}{r-\abs{s-c}}(7+\log^+\abs{f(c)})$$
    for all $s\in B_r(c)$, where $\log^+x=\max\{0,\log x\}$ for $x>0$.
\end{theorem}

Using this we can prove the following analogous result to the theorem of Brevig and Kouroupis.

\begin{theorem}\label{thm:strong-bohr}
    Let $f$ be an analytic function on a half-plane $\bC_\kappa$ such that $f(\bC_\kappa)\subseteq\bC\setminus\{\alpha,\beta\}$ for some complex numbers $\alpha\neq\beta$, and suppose $f$ is almost periodic on $\bC_\theta$ for some $\theta>\kappa$. Then $f$ is almost periodic on $\bC_\nu$ for all $\nu>\kappa$.
\end{theorem}

\begin{proof}
    Assume without loss of generality that $\alpha=0$ and $\beta=1$. As $f$ is almost periodic on $\bC_\theta$, it is bounded there, and so $M=\sup_{s\in\bC_\theta}\abs{f(s)}$ is finite. Fix $\nu\in(\kappa,\theta)$. We show that $f$ is bounded on $\bC_\nu$. Take any $s=\sigma+it\in\bC_\nu\setminus\bC_\theta$, and apply \cref{thm:schottky} with $c=\theta+it$ and $r=\theta-\kappa$ to obtain
    $$\log\abs{f(s)}\leq\frac{2\theta-(\kappa+\sigma)}{\sigma-\kappa}(7+\log^+\abs{f(\theta+it)})\leq\frac{2(\theta-\kappa)}{\nu-\kappa}(7+\log^+ M).$$
    From this it follows that $f$ is bounded on $\bC_\nu$ for all $\nu>\kappa$, so by \cref{cor:bohr-analogue} it follows that $f$ is almost periodic on $\bC_\nu$ for all $\nu>\kappa$.
\end{proof}

We now prove \cref{thm:strong-montel}, which we restate for convenience.

\strongmontel*

\begin{proof}
    That \labelcref{item:jointly-ap-somewhere} and \labelcref{item:jointly-ap-everywhere} are equivalent follows from \cref{thm:extend-joint-ap}, and so we only show that \labelcref{item:jointly-ap-everywhere} and \labelcref{item:strong-montel} are equivalent.
    
    Suppose first that \labelcref{item:jointly-ap-everywhere} holds. Let $\{f_n\}_{n\in\bZ^+}$ be a sequence in $\cF$. By the classical Montel theorem (e.g. \cite{rudin_real_1987}*{Theorem 14.6}) we can find an $f\in H^\infty(\bC_0)$ and a subsequence that converges to $f$ uniformly on all compact subsets of $\bC_0$. By replacing our original sequence with this subsequence, we may assume it is $\{f_n\}_{n\in\bZ^+}$ that converges to $f$ in this manner. Let $\kappa>0$. We show first that $\{f_n\}_{n\in\bZ^+}$ is uniformly Cauchy on the line $\Re s=\kappa$. Set $\cF'=\{f\vert_{\bC_{\kappa/2}}:f\in\cF\}$, and let $\varepsilon>0$. As the sequence converges uniformly on compact sets, we can find some $N\in\bZ^+$ such that
    \begin{equation}\label{eq:uniformly-cauchy-inteval}
        \max_{-d_{\cF'}(\varepsilon)\leq t\leq d_{\cF'}(\varepsilon)}\abs{f_n(\kappa+it)-f_m(\kappa+it)}\leq\varepsilon
    \end{equation}
    for all $n,m\geq N$. Set
    $$I_\ell=\left[\left(\ell-\frac{1}{2}\right)d_{\cF'}(\varepsilon),\left(\ell+\frac{1}{2}\right)d_{\cF'}(\varepsilon)\right]$$
    for $\ell\in\bZ$ and note that $\bR=\bigcup_{\ell\in\bZ}I_\ell$. Take any $n,m\geq N$ and any $s=\kappa+it$. Choose $\ell\in\bZ$ such that $t\in I_\ell$, and let $\tau_\ell\in I_{-\ell}\cap E_{\cF'}(\varepsilon)$, which can be done as $I_{-\ell}$ is a closed interval of length $d_{\cF'}(\varepsilon)$ and $\cF'$ is jointly almost periodic by assumption. A simple computation then shows that
    $$-d_{\cF'}(\varepsilon)\leq t+\tau_\ell\leq d_{\cF'}(\varepsilon).$$
    This allows us to estimate
    \begin{align*}
    \abs{f_n(s)-f_m(s)}
        &\leq\underbrace{\abs{V_{\tau_{\ell}}f_n(s)-f_n(s)}}_{\leq\varepsilon\text{ as }\tau_\ell\in E_{\cF'}(\varepsilon)}+\underbrace{\abs{V_{\tau_\ell}f_n(s)-V_{\tau_\ell}f_m(s)}}_{\leq\varepsilon\text{ by \cref{eq:uniformly-cauchy-inteval}}}+\underbrace{\abs{V_{\tau_\ell}f_m(s)-f_m(s)}}_{\leq\varepsilon\text{ as }\tau_\ell\in E_{\cF'}(\varepsilon)} \\
        &\leq3\varepsilon.
    \end{align*}
    From this it follows that $\{f_n\}_{n\in\bZ^+}$ is uniformly Cauchy in the line $\Re s=\kappa$, and hence uniformly convergent to $f$ there. Applying the maximum modulus principle for half-planes (e.g. \cite{bak_complex_2010}*{Theorem 15.1}) to the functions $f_n-f\in H^\infty(\bC_0)$, this extends to uniform convergence on $\bC_\kappa$. As $\kappa>0$ was arbitrary, the claim follows.

    Suppose now instead that \labelcref{item:strong-montel} holds. Let $\kappa>0$, set $\cF_\kappa=\{f\vert_{\bC_\kappa}:f\in\cF\}$, and let $\varepsilon>0$. By assumption we know that $\cF_\kappa$ is precompact in the uniform norm, and consequently it is totally bounded (e.g. \cite{willard_general_1970}*{Problem 24B}). As such we can find some $f_1,\dots,f_n\in\cF_\kappa$ with the property that if $f\in\cF_\kappa$, then
    \begin{equation}\label{eq:total-bound}
        \sup_{s\in\bC_\kappa}\abs{f(s)-f_k(s)}\leq\frac{\varepsilon}{3}
    \end{equation}
    for some $1\leq k\leq n$. Set $E(\varepsilon/3)=E_{f_1}(\varepsilon/3)\cap\dots\cap E_{f_n}(\varepsilon/3)$. We claim that $E(\varepsilon/3)\subseteq E_{\cF_\kappa}(\varepsilon)$. Indeed let $\tau\in E(\varepsilon/3)$ and $f\in\cF_\kappa$. Choose $1\leq k\leq n$ such that \cref{eq:total-bound} holds. Then, for any $s\in\bC_\kappa$, we have that
    $$\abs{V_\tau f(s)-f(s)}\leq\abs{V_\tau f(s)-V_\tau f_k(s)}+\abs{V_\tau f_k(s)-f_k(s)}+\abs{f_k(s)-f(s)}\leq\frac{\varepsilon}{3}+\frac{\varepsilon}{3}+\frac{\varepsilon}{3}=\varepsilon.$$
    The inclusion $E(\varepsilon/3)\subseteq E_{\cF_\kappa}(\varepsilon)$ follows, and as the former set is relatively dense by \cref{thm:finite-joint-ap}, so is the latter. As $\varepsilon>0$ and $\kappa>0$ were arbitrary, the result follows.
\end{proof}

With this, our Montel theorem for $H^\infty_\ap(\bC_0)$ is proven. As was mentioned in the introduction, there are many examples of uniformly bounded families in $H^\infty_\ap(\bC_0)$ that do not have a subsequence converging uniformly on half-planes.

\begin{example}\label{ex:no-uniformly-convergent-subsequence}
    Let $\{\lambda_n\}_{n\in\bZ^+}$ be any sequence of distinct non-negative real numbers that converges to some $\lambda\in[0,\infty)$, and consider the sequence $\{f_n\}_{n\in\bZ^+}$ in $H^\infty_\ap(\bC_0)$ defined by
    $$f_n(s)=e^{-\lambda_n s}.$$
    This sequence is clearly uniformly bounded and converges pointwise to $f(s)=e^{-\lambda s}$. Take now any $s=\sigma+it\in\bC_0$ and observe that
    $$\abs{f_n(s)-f(s)}^2=e^{-2\lambda_n\sigma}+e^{-2\lambda\sigma}-2e^{-(\lambda_n+\lambda)\sigma}\cos((\lambda_n-\lambda)t).$$
    In particular, taking $t_n=\pi/(\lambda_n-\lambda)$, we see that
    $$\sup_{\Re s=\kappa}\abs{f_n(s)-f(s)}^2\geq\abs{f_n(\kappa+it_n)-f(\kappa+it_n)}^2=e^{-2\lambda_n\kappa}+e^{2\lambda\kappa}+2e^{-(\lambda_n+\lambda)\kappa},$$
    from which we get that
    $$\liminf_{n\to\infty}\sup_{\Re s=\kappa}\abs{f_n(s)-f(s)}\geq2e^{-\lambda\kappa}.$$
    In particular, no subsequence of $\{f_n\}_{n\in\bZ^+}$ can converges uniformly even on a vertical line $\Re s=\kappa$, and so consequently also not on a half-plane.
\end{example}

We will use \cref{thm:strong-montel} to prove several Montel-type theorems for subspaces of $H^\infty_\ap(\bC_0)$ by establishing joint almost periodicity.

Our first application will concern spaces analogous to $\sH^\infty$. By a \emph{frequency} we shall mean an increasing sequence $\lambda=\{\lambda_n\}_{n\in\bZ^+}$ of non-negative real numbers with limit $\infty$. Given a frequency $\lambda$, we define $\sD_\ext^\infty(\lambda)$ to be the space of all $f\in H^\infty(\bC_0)$ that can be written as a convergent $\lambda$-Dirichlet series
\begin{equation}\label{eq:lambda-dirichlet}
    f(s)=\sum_{n=1}^\infty a_ne^{-\lambda_ns}
\end{equation}
on some half-plane $\bC_\kappa$ with $\kappa>0$. Observe that $\sH^\infty$ corresponds to the particular choice $\lambda=\{\log n\}_{n\in\bZ^+}$. Associated to a frequency $\lambda$ we have the number $L(\lambda)$, which we define to be the abscissa of convergence of the $\lambda$-Dirichlet series
\begin{equation}\label{eq:l-dirichlet}
    \sum_{n=1}^\infty e^{-\lambda_n s},
\end{equation}
i.e., $L(\lambda)$ is the unique value in $[0,\infty]$ such that if $s=\sigma+it$, then the series in \cref{eq:l-dirichlet} diverges for $\sigma<L(\lambda)$ and converges for $\sigma>L(\lambda)$ (or converges nowhere if $L(\lambda)=\infty$).

In \cite{schoolmann_bohrs_2020}, I. Schoolmann showed that if $L(\lambda)=0$, then we have a Montel-type theorem for $\sD_\ext^\infty(\lambda)$. The approach of Schoolmann is based on showing that if $L(\lambda)=0$, then Bohr's theorem holds for $\sD_\ext^\infty(\lambda)$ in the sense that the $\lambda$-Dirichlet series expansion \cref{eq:lambda-dirichlet} of any $f\in\sD_\ext^\infty(\lambda)$ converges uniformly on all half-planes $\bC_\kappa$ with $\kappa>0$. We remark that there are frequencies $\lambda$ such that Bohr's theorem does not hold for $\sD_\ext^\infty(\lambda)$. Explicit constructions of such $\lambda$ can be found in e.g. \cites{neder_zum_1922,schoolmann_bohrs_2020}.

Using \cref{thm:strong-montel} we show that the analogue of Bayart's strong Montel theorem holds for $\sD_\ext^\infty(\lambda)$ under the assumption that $L(\lambda)<\infty$. In our approach, we completely avoid the need for Bohr's theorem, and only need that we can bound the coefficients in \cref{eq:lambda-dirichlet} by the norm of $f$. For this, we use the following theorem of Schnee.

\begin{theorem}[Schnee]
    Let $\lambda$ be a frequency, and let $f$ be given by \cref{eq:lambda-dirichlet} on some half-plane $\bC_\kappa$. Then
    $$a_n=\lim_{n\to\infty}\frac{1}{2T}\int_{-T}^Tf(\sigma+it)e^{\lambda_n(\sigma+it)}\,\dd t$$
    for all $n\in\bZ^+$ and all $\sigma>\kappa$.
\end{theorem}

A proof of this theorem can be found in \cite{helson_convergent_1963}. A simple argument using the Cauchy integral theorem similar to the proof of \cite{queffelec_diophantine_2020}*{Theorem 6.1.1} yields the following corollary.

\begin{corollary}\label{cor:coefficient-estimate}
    Let $\lambda$ be a frequency, and let $f\in\sD_\ext^\infty(\lambda)$ be given by \cref{eq:lambda-dirichlet} on some half-plane $\bC_\kappa$. Then
    $$\abs{a_n}\leq\norm{f}_\infty$$
    for all $n\in\bZ^+$.
\end{corollary}

We now prove our Montel theorem for $\sD_\ext^\infty(\lambda)$ with $L(\lambda)<\infty$.

\begin{theorem}\label{thm:l-lambda-montel}
    Let $\lambda$ be a frequency satisfying $L(\lambda)<\infty$, and let $\cF$ be a uniformly bounded family in $\sD_\ext^\infty(\lambda)$. For every sequence $\{f_n\}_{n\in\bZ^+}$ in $\cF$ there exists a function $f\in\sD_\ext^\infty(\lambda)$ and a subsequence $\{f_{n_k}\}_{k\in\bZ^+}$ that converges uniformly to $f$ on $\bC_\kappa$ for all $\kappa>0$.
\end{theorem}

\begin{proof}
    Fix $\kappa>L(\lambda)$. The goal is to show that $\cF$ is jointly almost periodic on $\bC_\kappa$ and apply \cref{thm:strong-montel}. Let $\varepsilon>0$. For $f\in\cF$, write $f(s)=\sum_{n=1}^\infty a_n(f)e^{-\lambda_n s}$ for $s\in\bC_\kappa$; it is clear that the series converges absolutely and uniformly to $f$ on $\bC_\kappa$ as a consequence of \cref{cor:coefficient-estimate}. Set $M=\sup_{f\in\cF}\norm{f}_\infty$ and note that
    $$\sup_{\substack{n\in\bZ^+ \\f\in\cF}}\abs{a_n(f)}\leq M$$
    by \cref{cor:coefficient-estimate}. Then, for all $f\in\cF$, $\tau\in\bR$, $s\in\bC_\kappa$ and $N\in\bZ^+$, we can estimate
    \begin{align}\label{eq:series-estimate}
    \begin{split}
    \abs{V_\tau f(s)-f(s)}
        &\leq M\sum_{n=1}^N\lvert e^{-\lambda_n(s+i\tau)}-e^{-\lambda_ns}\rvert+M\sum_{n=N+1}^\infty e^{-\lambda_n\Re s}\lvert e^{-\lambda_ni\tau}-1\rvert \\
        &\leq M\sum_{n=1}^N\lvert e^{-\lambda_n(s+i\tau)}-e^{-\lambda_ns}\rvert+2M\sum_{n=N+1}^\infty e^{-\lambda_n\kappa}.
    \end{split}
    \end{align}
    Choose $N\in\bZ^+$ such that
    $$\sum_{n=N+1}^\infty e^{-\lambda_n\kappa}\leq\frac{\varepsilon}{4M},$$
    which can be done as $\kappa>L(\lambda)$. Take now any $\tau\in\bR$ which is a common $\frac{\varepsilon}{2MN}$-translation number for the maps $s\mapsto e^{-\lambda_n s}$, $n=1,\dots,N$ on $\bC_\kappa$. Using \cref{eq:series-estimate} we can then estimate
    $$\abs{V_\tau f(s)-f(s)}\leq M\sum_{n=1}^N\frac{\varepsilon}{2MN}+2M\frac{\varepsilon}{4M}=\varepsilon$$
    for all $f\in\cF$ and all $s\in\bC_\kappa$. Consequently $\tau$ is a common $\varepsilon$-translation number of $\cF$ on $\bC_\kappa$. As this holds for all $\tau$ in a relatively dense set by \cref{thm:finite-joint-ap}, it now follows that $\cF$ is jointly almost periodic on $\bC_\kappa$. The result now follows by applying \cref{thm:strong-montel} and noting that the limit function must be in $\sD_\ext^\infty(\lambda)$ as a consequence of \cref{cor:coefficient-estimate}.
\end{proof}

If $\lambda=\{\log n\}_{n\in\bZ^+}$, then $L(\lambda)=1$, and we recover Bayart's strong Montel theorem \cite{bayart_hardy_2002}*{Lemma 18}. Similarly we also obtain the case of Schoolmann's strong Montel theorem \cite{schoolmann_bohrs_2020}*{Theorem 4.10} for when $L(\lambda)=0$.

The goal of the above is to illustrate the power of \cref{thm:strong-montel}; indeed the proof above is rather elementary and uses only standard results on almost periodic functions. As was mentioned in the introduction, a stronger generalization of Bayart's Montel theorem was proven by Defant, Vidal, Schoolmann, and Sevilla-Peris in \cite{defant_frechet_2021}, and we now turn our attention to showing their result as a consequence of \cref{thm:strong-montel}. This requires a bit more machinery than the proof of \cref{thm:l-lambda-montel}.

For a function $f\in H^\infty_\ap(\bC_0)$, the \emph{Bohr coefficients} are defined as the numbers
$$a_\lambda(f)=\lim_{T\to\infty}\frac{1}{2T}\int_{-T}^T f(\sigma+it)e^{\lambda(\sigma+it)}\,\dd t$$
with $\lambda\geq0$ and $\sigma>0$. An application of the Cauchy integral theorem shows that the definition of $a_\lambda(f)$ is independent of the choice of $\sigma>0$. As a simple consequence of the triangle inequality we also obtain the following analogue of \cref{cor:coefficient-estimate}.

\begin{lemma}\label{lemma:bohr-coefficient-estimate}
    Let $f\in H^\infty_\ap(\bC_0)$. Then
    $$\abs{a_\lambda(f)}\leq\norm{f}_\infty$$
    for all $\lambda\geq0$.
\end{lemma}

We define also the \emph{Bohr spectrum} of $f$ as the set
$$\Lambda(f)=\{\lambda\geq0:a_\lambda(f)\neq0\}.$$
A standard result in the theory of almost periodic functions is that the Bohr spectrum is at most countable (see e.g. \cite{besicovitch_almost_1955}*{Chapter 1, §3, Theorem 5}). For a frequency $\lambda$, we define $H^\infty_\lambda(\bC_0)$ as the space of all $f\in H^\infty_\ap(\bC_0)$ with $\Lambda(f)$ contained in $\lambda$. The goal is to show that the direct analogue of Bayart's strong Montel theorem holds for $H^\infty_\lambda(\bC_0)$.

To a function $f\in H^\infty_\ap(\bC_0)$ we formally associate a general Dirichlet series, which we write as
$$f(s)\sim\sum_{\lambda\in\Lambda(f)}a_\lambda(f)e^{-\lambda s}.$$
Note that, in contrast to the spaces $\sD_\ext^\infty(\lambda)$, there is no assumption of convergence of the general Dirichlet series of $f$, and as such the approach taken in the proof of \cref{thm:l-lambda-montel} no longer works for general $f\in H^\infty_\ap(\bC_0)$. There is a way around this, however. Given a frequency $\lambda=\{\lambda_n\}_{n\in\bZ^+}$ and a function $f\in H^\infty_\lambda(\bC_0)$, we define the \emph{Riesz means} of $f$ by
\begin{equation}\label{eq:riesz-means}
    R_\omega f(s)=\sum_{\lambda_n<\omega}a_{\lambda_n}(f)\left(1-\frac{\lambda_n}{\omega}\right)e^{-\lambda_n s}
\end{equation}
for $\omega>0$. Since $\lambda$, by definition, has limit $\infty$, we see that the sum in \cref{eq:riesz-means} has only finitely many terms. In our setting, the Riesz means will take the role of the partial sums in the proof of \cref{thm:l-lambda-montel}, and as such we need the following approximation lemma.

\begin{lemma}\label{lemma:riesz-mean-approximation}
    Let $\lambda$ be a frequency and let $\cF$ be a uniformly bounded family in $H^\infty_\lambda(\bC_0)$. Then
    $$\lim_{\omega\to\infty}\sup_{\substack{s\in\bC_\kappa \\ f\in\cF}}\abs{R_\omega f(s)-f(s)}=0$$
    for all $\kappa>0$.
\end{lemma}

\begin{proof}
    Fix $\kappa>0$. By \cite{defant_frechet_2021}*{Lemma 3.3} we can write
    $$R_\omega f(\kappa+\sigma+it)=(f_\sigma*P_\kappa*K_\omega)(t)$$
    for $\sigma+it\in\bC_0$, where $f_\sigma(t)=f(\sigma+it)$, $P_\kappa(t)=\frac{1}{\pi}\frac{\kappa}{\kappa^2+t^2}$ is the Poisson kernel of the right half-plane, and $\{K_\omega\}_{\omega>0}$ is an approximate identity, meaning that $g*K_\omega\to g$ in $L^1(\bR)$ as $\omega\to\infty$ for all $g\in L^1(\bR)$. Set $M=\sup_{f\in\cF}\norm{f}_\infty$ and take $s=\kappa+\sigma+it\in\bC_\kappa$, $f\in\cF$ and $\omega>0$. As $f(s)=(f_\sigma*P_\kappa)(t)$ by the Poisson integral formula (e.g. \cite{duren_theory_1970}*{Theorem 11.2}), we can estimate
    $$\abs{R_\omega f(s)-f(s)}=\abs{(f_\sigma*(P_\kappa*K_\omega-P_\kappa))(t)}\leq M\norm{P_\kappa*K_\omega-P_\kappa}_{L^1(\bR)}.$$
    The result now follows by taking a supremum over $s\in\bC_\kappa$, $f\in\cF$, and letting $\omega\to\infty$.
\end{proof}

With this we can prove the result of Defant, Vidal, Schoolmann, and Sevilla-Peris.

\begin{theorem}[Defant--Vidal--Schoolmann--Sevilla-Peris]\label{thm:defant-montel}
    Let $\lambda$ be a frequency, and let $\cF$ be a uniformly bounded family in $H^\infty_\lambda(\bC_0)$. For every sequence $\{f_n\}_{n\in\bZ^+}$ in $\cF$ there exists a function $f\in H^\infty_\lambda(\bC_0)$ and a subsequence $\{f_{n_k}\}_{k\in\bZ^+}$ that converges uniformly to $f$ on $\bC_\kappa$ for all $\kappa>0$.
\end{theorem}

\begin{proof}
    Let $\kappa>0$. The goal is to show that $\cF$ is jointly almost periodic on $\bC_\kappa$ and apply \cref{thm:strong-montel}. Let $\varepsilon>0$. By \cref{lemma:riesz-mean-approximation} we can choose $\omega>0$ such that
    $$\sup_{\substack{s\in\bC_\kappa \\f\in\cF}}\abs{R_\omega f(s)-f(s)}\leq\frac{\varepsilon}{3}.$$
    Let $N_\omega$ denote the cardinality of the finite set $\{\lambda_n:\lambda_n<\omega\}$. Take any $\tau\in\bR$ that is a common $\frac{\varepsilon}{3MN_\omega}$-translation number for the maps $s\mapsto e^{-\lambda_n s}$ for $\lambda_n<\omega$ on $\bC_\kappa$. Using \cref{lemma:bohr-coefficient-estimate}, we can estimate
    \begin{align*}
    \abs{V_\tau R_\omega f(s)-R_\omega f(s)}
        &\leq\sum_{\lambda_n<\omega}\abs{a_{\lambda_n}(f)}\left(1-\frac{\lambda_n}{\omega}\right)\lvert e^{-\lambda_n(s+i\tau)}-e^{-\lambda_ns}\rvert \\
        &\leq\sum_{\lambda_n<\omega}M\frac{\varepsilon}{3MN_\omega} \\
        &=\frac{\varepsilon}{3}
    \end{align*}
    for all $f\in\cF$ and $s\in\bC_\kappa$. This allows us then to estimate
    \begin{align*}
    \abs{V_\tau f(s)-f(s)}
        &\leq\abs{V_\tau f(s)-V_\tau R_\omega f(s)}+\abs{V_\tau R_\omega f(s)-R_\omega f(s)}+\abs{R_\omega f(s)-f(s)} \\
        &\leq\frac{\varepsilon}{3}+\frac{\varepsilon}{3}+\frac{\varepsilon}{3} \\
        &=\varepsilon
    \end{align*}
    for all $f\in\cF$ and $s\in\bC_\kappa$. It follows that $\tau$ is a common $\varepsilon$-translation number of $\cF$ on $\bC_\kappa$. As this holds for all $\tau$ in a relatively dense set by \cref{thm:finite-joint-ap}, it follows that $\cF$ is jointly almost periodic on $\bC_\kappa$. The result now follows by applying \cref{thm:strong-montel}; the limit function must clearly lie in $H^\infty_\lambda(\bC_0)$.
\end{proof}

\section{The almost periodic half-plane algebra}\label{sec:ap-algebra}

In this section we will consider the space $A_\ap(\bC_0)$ defined in the introduction. Our main goal in this section is to prove that $A_\ap(\bC_0)$ is precisely the closure of general Dirichlet polynomials (\cref{thm:ap-algebra-dirichlet} above). The idea of the proof is to extend $f\in A_\ap(\bC_0)$ to $\overline\bC_0$, show that $f$ is almost periodic on the imaginary axis, and conclude using standard approximation results about almost periodic functions on the real line. As such, the first step is the following lemma.

\begin{lemma}\label{lemma:extension-is-ap}
    Let $f$ be the unique uniformly continuous extension to $\overline\bC_0$ of some function in $A_\ap(\bC_0)$. Then $f$ is almost periodic on $\overline\bC_0$.
\end{lemma}

\begin{proof}
    Let $\varepsilon>0$. As $f$ is uniformly continuous, we can find some $\delta>0$ such that if $s_1,s_2\in\overline\bC_0$ satisfy $\abs{s_1-s_2}\leq\delta$, then
    $$\abs{f(s_1)-f(s_2)}\leq\frac{\varepsilon}{3}.$$
    Let now $\tau\in\bR$ be any $\varepsilon/3$-translation number of $f$ on $\overline\bC_\delta$. Then, for any $s\in\overline\bC_0$, we have that
    \begin{align*}
    \abs{V_\tau f(s)-f(s)}
        &\leq\abs{V_\tau f(s)-V_\tau f(s+\delta)}+\abs{V_\tau f(s+\delta)-f(s+\delta)}+\abs{f(s+\delta)-f(s)} \\
        &\leq\frac{\varepsilon}{3}+\frac{\varepsilon}{3}+\frac{\varepsilon}{3} \\
        &=\varepsilon.
    \end{align*}
    It follows that $\tau$ is an $\varepsilon$-translation number of $f$ on $\overline\bC_0$. As $\tau$ was an arbitrary element of a relatively dense set, the result follows.
\end{proof}

With this we can prove \cref{thm:ap-algebra-dirichlet}, which we restate for convenience.

\apalgebradirichlet*

\begin{proof}
    It is clear that $A_\ap(\bC_0)$ is a closed subspace of $H^\infty_\ap(\bC_0)$ containing all general Dirichlet polynomials, from which the inclusion $\overline{\sP_D}\subseteq A_\ap(\bC_0)$ follows. For the reverse inclusion, let $f\in A_\ap(\bC_0)$ and let $\varepsilon>0$. By replacing $f$ with its unique uniformly continuous extension to $\overline\bC_0$, we may assume $f$ is defined on $\overline\bC_0$. By \cref{lemma:extension-is-ap}, we then know that $f$ is almost periodic on the imaginary axis $\partial\bC_0$. Using \cite{besicovitch_almost_1955}*{Chapter 1, §9, Theorem 4} and \cite{besicovitch_almost_1955}*{Chapter 3, §3, Theorem 8}, we can then find a sequence $\{P_n\}_{n\in\bZ^+}$ of general Dirichlet polynomials that converges uniformly to $f$ on $\partial\bC_0$. Applying the maximum modulus principle for half-planes (e.g. \cite{bak_complex_2010}*{Theorem 15.1}) to the functions $f-P_n$, this extends to uniform convergence on $\bC_0$. Consequently $f\in\overline{\sP_D}$. The result follows.
\end{proof}

\begin{remark}
    One can use the exact same argument to show that $\sA(\bC_0)$ is precisely the closure of the set of all Dirichlet polynomials in the uniform norm \cite{aron_dirichlet_2017}*{Theorem 2.3}. Indeed if $f\in\sH^\infty$, then by \cite{besicovitch_almost_1955}*{Chapter 3, §3, Theorem 1} the approximating sequence obtained from \cite{besicovitch_almost_1955}*{Chapter 1, §9, Theorem 4} can be chosen to consist of ordinary Dirichlet polynomials.
\end{remark}

Another interesting property of $A_\ap(\bC_0)$ is that it is not separable, contrasting the classical disk algebra.

\begin{theorem}\label{thm:algebra-not-separable}
    $A_\ap(\bC_0)$ is not separable.
\end{theorem}

\begin{proof}
    Consider the set $P=\{f_\lambda:\lambda\geq0\}$, where $f_\lambda(s)=e^{-\lambda s}$. We claim that $f_\lambda\notin\overline{P\setminus\{f_\lambda\}}$ for all $\lambda\geq0$. Suppose for a contradiction that $f_\lambda\in\overline{P\setminus\{f_\lambda\}}$ for some $\lambda\geq0$. Then we can find a sequence $\{\lambda_n\}_{n\in\bZ^+}$ in $[0,\infty)\setminus\{\lambda\}$ such that $f_{\lambda_n}\to f_\lambda$ uniformly as $n\to\infty$. Then in particular $e^{-\lambda_n}=f_{\lambda_n}(1)\to f_\lambda(1)=e^{-\lambda}$, so that $\lambda_n\to\lambda$ as $n\to\infty$. But then by the argument in \cref{ex:no-uniformly-convergent-subsequence}, we know that $\{f_{\lambda_n}\}_{n\in\bZ^+}$ cannot converge uniformly to $f_\lambda$; a contradiction. To see that this implies that $A_\ap(\bC_0)$ is not separable, suppose for a contradiction that it was. As a subset of a separable metric space is separable (e.g. \cite{willard_general_1970}*{Problem 16G}), it then follows that $P$ is separable. But $P$ cannot be separable. Indeed suppose it had a countable dense subset $A$. As $P$ is uncountable we can take $f_\lambda\in P\setminus A$. But then we would have that $f_\lambda\in\overline A\subseteq\overline{P\setminus\{f_\lambda\}}$; a contradiction. It follows that $A_\ap(\bC_0)$ is not separable.
\end{proof}

In \cite{contreras_composition_2024}*{Proposition 2.9}, the authors construct a function $f\in\sH^\infty$ that has a continuous extension to $\overline\bC_0$ but that is not uniformly continuous on $\bC_0$ (see also the example following \cite{brevig_almost_2025}*{Lemma 2.2}). This example shows that the inclusion $A_\ap(\bC_0)\subseteq H^\infty_\ap(\bC_0)$ is strict. One can also see this by noting that if $g$ is a bounded analytic function on the open unit disc $\bD$ such that the limit $\lim_{z\to e^{i\theta}}g(z)$ does not exist for some $e^{i\theta}$ on the circle, then $f(s)=g(e^{-\lambda s})$ defines a function in $H^\infty_\ap(\bC_0)\setminus A_\ap(\bC_0)$.

\section{Boundedness and compactness of composition operators}\label{sec:composition-operators}

Our first goal in this section is to characterize the symbols $\varphi:\bC_0\to\bC_0$ inducing bounded composition operators on $H^\infty_\ap(\bC_0)$. Clearly this is the same as asking for which $\varphi$ it holds that $f\circ\varphi\in H^\infty_\ap(\bC_0)$ for all $f\in H^\infty_\ap(\bC_0)$. We will show that these are precisely the symbols of the form
\begin{equation}\label{eq:symbol-form}
    \varphi(s)=as+\psi(s)
\end{equation}
with $a\geq0$ and $\psi$ an analytic function that is almost periodic on $\bC_\kappa$ for all $\kappa>0$. We shall start by proving that if $\varphi$ is of the form \cref{eq:symbol-form}, then it induces a bounded composition operator on $H^\infty_\ap(\bC_0)$. Our approach is inspired by the approach of Bayart in his characterization of the bounded composition operators on $\sH^\infty$ (see \cites{bayart_hardy_2002,bayart_composition_2021}).

An important tool for our approach is the concept of a vertical limit function of an almost periodic function. A standard result in the theory of almost periodic functions states that a function $f$ is almost periodic if and only if for any sequence $\{\tau_n\}_{n\in\bZ^+}$ of real numbers, the sequence $\{V_{\tau_n}f\}_{n\in\bZ^+}$ of vertical translations of $f$ has a uniformly convergent subsequence (see e.g. \cite{besicovitch_almost_1955}*{Chapter 1, §2, Theorems 3-4}). The functions which arise as uniform limits of vertical translations of $f$ are called \emph{vertical limit functions} of $f$. We will need the following lemma on the images of half-planes under vertical limit functions in the analytic setting, the proof of which can be done in more or less the same way as \cite{brevig_norms_2020}*{Lemma 1}.

\begin{lemma}\label{lemma:image-vertical-limit}
    Let $f$ be an analytic almost periodic function on a half-plane $\bC_\kappa$, and let $g$ be a vertical limit function of $f$. Then
    $$f(\bC_\alpha)=g(\bC_\alpha)$$
    for all $\alpha\geq\kappa$.
\end{lemma}

We will also need the following standard fact about the behavior at $+\infty$ for analytic almost periodic functions (see e.g. \cite{bochner_potential-theoretic_1963}*{Theorem 1} or \cite{besicovitch_almost_1955}*{Chapter 3, §4, Theorem 1} for a proof).

\begin{lemma}\label{lemma:limit-at-infinity}
    If $f$ is an analytic almost periodic function on a half-plane $\bC_\kappa$, then there exists a complex number, denoted by $f(+\infty)$, such that
    $$\lim_{\kappa'\to\infty}\sup_{s\in\bC_{\kappa'}}\abs{f(s)-f(+\infty)}=0.$$
\end{lemma}

With this we can prove two important lemmas for our theorem, which play a similar role to \cite{gordon_composition_1999}*{Proposition 4.2} in the setting of spaces of Dirichlet series.

\begin{lemma}\label{lemma:psi-image}
    Let $\psi$ be an analytic function on $\bC_0$ that is almost periodic on $\bC_\kappa$ for all $\kappa>0$. If $\psi(\bC_0)\subseteq\bC_0$, then, for all $\kappa>0$, there exists a $\nu>0$ such that $\psi(\bC_\kappa)\subseteq\bC_\nu$.
\end{lemma}

\begin{proof}
    If $\psi$ is constant then the claim is trivial, so suppose $\psi$ is not constant. Fix $\kappa>0$ and suppose for a contradiction that there exists some sequence $\{\sigma_n+it_n\}_{n\in\bZ^+}$ in $\bC_\kappa$ such that $\Re\psi(\sigma_n+it_n)\to0$ as $n\to\infty$. By passing to a subsequence, we may also suppose that $\{\sigma_n\}_{n\in\bZ^+}$ converges to some $\sigma'\in[\kappa,+\infty]$ and that $\{V_{t_n}\psi\}_{n\in\bZ^+}$ converges uniformly to some function $g$ on $\bC_{\kappa/2}$. We claim that $g$ is constant. Observe first that
    \begin{equation}\label{eq:g-at-sigma}
        \Re g(\sigma')=\lim_{n\to\infty}\Re V_{t_n}\psi(\sigma_n)=\lim_{n\to\infty}\Re\psi(\sigma_n+it_n)=0.
    \end{equation}
    Consider now the map $G=e^{-g}$, which is analytic on $\bC_{\kappa/2}$. As $g(\bC_{\kappa/2})=\psi(\bC_{\kappa/2})\subseteq\bC_0$ by \cref{lemma:image-vertical-limit}, we have that $G(\bC_{\kappa/2})\subseteq\bD$, where $\bD$ denotes the open unit disk. Now if $\sigma'\neq+\infty$, then
    $$\abs{G(\sigma')}=e^{-\Re g(\sigma')}=1$$
    by \cref{eq:g-at-sigma}, which contradicts that $G(\bC_{\kappa/2})\subseteq\bD$, so it must be the case that $\sigma'=+\infty$. Set now
    $$M_G(\sigma)=\sup_{t\in\bR}\abs{G(\sigma+it)}$$
    for $\sigma\geq\kappa$. By the maximum modulus principle for half-planes (e.g. \cite{bak_complex_2010}*{Theorem 15.1}), together with \cref{lemma:limit-at-infinity} we then have that
    $$1\geq M_G(\kappa)\geq \lim_{\sigma\to\infty}M_G(\sigma)=e^{-\Re g(+\infty)}=1.$$
    But then $\inf_{t\in\bR}\Re g(\kappa+it)=0$, so we can proceed with an argument similar to the case when $\sigma'\neq+\infty$ to obtain a contradiction. We have thus shown that $g$ is constant, and so by \cref{lemma:image-vertical-limit} it also follows that $\psi$ is constant; a contradiction. The result follows.
\end{proof}

\begin{lemma}\label{lemma:psi-image-cn}
    Let $\varphi:\bC_0\to\bC_0$ be of the form $\varphi(s)=as+\psi(s)$ with $a\geq0$ and $\psi$ an analytic function that is almost periodic on $\bC_\kappa$ for all $\kappa>0$. If $\psi$ is not constant, then $\psi(\bC_0)\subseteq\bC_0$.
\end{lemma}

\begin{proof}
    Set $F=e^{-\psi}$, and note that $F$ is analytic on $\bC_0$. As $\psi$ is almost periodic on $\bC_\kappa$, it is bounded there, and consequently so is $F$, and for $s\in\bC_0\setminus\bC_\kappa$ we have that
    $$\abs{F(s)}=e^{-\Re\varphi(s)+a\Re s}\leq e^{a\kappa}.$$
    In particular $F\in H^\infty(\bC_0)$. If $s\in\bC_0$ and $\varepsilon\in(0,\Re s)$, then by the maximum modulus principle for half-planes (e.g. \cite{bak_complex_2010}*{Theorem 15.1}) we have that
    $$e^{-\Re\psi(s)}=\abs{F(s)}\leq\sup_{t\in\bR}\abs{F(\varepsilon+it)}\leq e^{a\varepsilon}.$$
    Letting $\varepsilon\downarrow0$ we obtain that $\Re\psi(s)\geq0$. As $\psi$ is non-constant, the open mapping property of analytic functions then finally implies that $\psi(\bC_0)\subseteq\bC_0$.
\end{proof}

The final lemma is a simple result on common translation numbers of almost periodic functions on possibly different half-planes.

\begin{lemma}\label{lemma:relatively-dense-intersection}
    Let $f$ be an almost periodic function on a half-plane $\bC_\kappa$ and let $g$ be an almost periodic function on a (possibly different) half-plane $\bC_{\kappa'}$. Then, for all $a>0$ and all $\varepsilon>0$, the set $E_f(\varepsilon)\cap aE_g(\varepsilon)$ is relatively dense.
\end{lemma}

\begin{proof}
    Define $g_a:\bC_\kappa\to\bC$ by
    $$g_a(s)=g\left(\frac{s-\kappa}{a}+\kappa'\right).$$
    Clearly $g_a$ is almost periodic, and so by \cref{thm:finite-joint-ap} we have that $E_f(\varepsilon)\cap E_{g_a}(\varepsilon)$ is relatively dense. The result now follows by observing that $E_{g_a}(\varepsilon)=aE_g(\varepsilon)$.
\end{proof}

With this we can prove that all $\varphi:\bC_0\to\bC_0$ of the form \cref{eq:symbol-form} induce bounded composition operators on $H^\infty_\ap(\bC_0)$.

\begin{lemma}\label{lemma:sufficiency}
    Let $\varphi:\bC_0\to\bC_0$ be of the form $\varphi(s)=as+\psi(s)$ with $a\geq0$ and $\psi$ an analytic function that is almost periodic on $\bC_\kappa$ for all $\kappa>0$. Then $f\circ\psi\in H^\infty_\ap(\bC_0)$ for all $f\in H^\infty_\ap(\bC_0)$.
\end{lemma}

\begin{proof}
    We consider only the case when $a>0$ and $\psi$ is non-constant; the remaining cases are similar and left to the reader. Fix $f\in H^\infty_\ap(\bC_0)$. That $f\circ\varphi\in H^\infty(\bC_0)$ is clear, and so we only need to show almost periodicity on some half-plane. Let $\kappa>0$. We show that $f\circ\varphi$ is almost periodic on $\bC_\kappa$. By \cref{lemma:psi-image} and \cref{lemma:psi-image-cn} there exists some $\nu>0$ such that $\psi(\bC_\kappa)\subseteq\bC_\nu$. As $f$ is uniformly continuous on $\bC_\nu$, we can find some $\delta>0$ such that if $s_1,s_2\in\bC_\nu$ satisfy $\abs{s_1-s_2}\leq\delta$, then
    \begin{equation}\label{eq:uniform-cont}
        \abs{f(s_1)-f(s_2)}\leq\frac{\varepsilon}{2}.
    \end{equation}
    Set $\varepsilon'=\min\{\delta,\varepsilon/2\}$ and let $\tau\in E_{\psi\vert_{\bC_\kappa}}(\varepsilon')\cap\frac{1}{a}E_{f\vert_{\bC_\nu}}(\varepsilon')$. Then $\tau\in E_{\psi\vert_{\bC_\kappa}}(\delta)$ and $a\tau\in E_{f\vert_{\bC_\nu}}(\varepsilon/2)$, so that we, for all $s\in\bC_\kappa$, can estimate
    \begin{align*}
    \abs{V_\tau(f\circ\varphi)(s)-(f\circ\varphi)(s)}
        &=\abs{V_{a\tau}f(as+V_\tau\psi(s))-f(as+\psi(s))} \\
        &\leq\underbrace{\abs{V_{a\tau}f(as+V_\tau\psi(s))-f(as+V_\tau\psi(s))}}_{\leq\varepsilon/2\text{ as }a\tau\in E_{f\vert_{\bC_\nu}}(\varepsilon/2)} \\
        &\qquad+\underbrace{\abs{f(as+V_\tau\psi(s))-f(as+\psi(s))}}_{\leq\varepsilon/2\text{ by \cref{eq:uniform-cont} as }\tau\in E_{\psi\vert_{\bC_\kappa}}(\delta)} \\
        &\leq\varepsilon.
    \end{align*}
    It follows that any such $\tau$ is an $\varepsilon$-translation number of $f\circ\psi$ on $\bC_\kappa$, and as $\tau$ was an arbitrary element of a relatively dense set by \cref{lemma:relatively-dense-intersection}, the result follows.
\end{proof}

We next turn to proving the converse, i.e., that all symbols inducing bounded composition operators on $H^\infty_\ap(\bC_0)$ must be of the form \cref{eq:symbol-form}. Our main tool for this is the following theorem due to Bochner; a proof can be found in \cite{bochner_potential-theoretic_1963}*{Theorem 2}.

\begin{theorem}[Bochner]\label{thm:bochner}
    Let $f$ be an analytic almost periodic function on some half-plane $\bC_\kappa$. If $f(s)\neq0$ for all $s\in\bC_\kappa$, then there exists a $\lambda\geq0$ and an analytic almost periodic function $g$ on $\bC_\kappa$ with $g(+\infty)\neq0$ such that
    $$f(s)=e^{-\lambda s}g(s)$$
    for all $s\in\bC_\kappa$.
\end{theorem}

We will also need the following lemma.

\begin{lemma}\label{lemma:ap-exponent}
    Let $\psi$ be an analytic function on $\bC_0$ and $\lambda>0$ such that $g=e^{-\lambda\psi}$ is almost periodic on $\bC_\kappa$ for some $\kappa>0$ and $g(+\infty)\neq0$. Then $\psi$ is almost periodic on $\bC_{\kappa'}$ for some $\kappa'>0$.
\end{lemma}

\begin{proof}
    Set $r=\abs{g(+\infty)}/2$ and use \cref{lemma:limit-at-infinity} to choose $\kappa'\geq\kappa$ such that
    $$e^{-\lambda\psi(s)}\in B_r(g(+\infty))$$
    for all $s\in\bC_{\kappa'}$, where $B_r(g(+\infty))$ denotes the open ball in $\bC$ of radius $r$ centered at $g(+\infty)$. We claim that $\psi$ is almost periodic on $\bC_{\kappa'}$. As $B_r(g(+\infty))$ is bounded away from zero, we can let $\log$ denote the principal branch of the logarithm defined on $B_r(g(+\infty))$. Observe then that $\psi+\frac{1}{\lambda}\log g$ is constant on $\bC_{\kappa'}$, by which we can estimate
    \begin{equation}\label{eq:estimate-psi-g}
        \abs{\psi(s_1)-\psi(s_2)}=\frac{1}{\lambda}\abs{\log g(s_1)-\log g(s_2)}=\frac{1}{\lambda}\abs{\int_{g(s_2)}^{g(s_1)}\frac{\dd\zeta}{\zeta}}\leq\frac{\abs{g(s_1)-g(s_2)}}{\lambda r}
    \end{equation}
    for all $s_1,s_2\in\bC_{\kappa'}$. It is now clear that the almost periodicity of $g$ on $\bC_{\kappa'}$ together with \cref{eq:estimate-psi-g} implies that $\psi$ is almost periodic on $\bC_{\kappa'}$.
\end{proof}

With this we can now prove that the form \cref{eq:symbol-form} is necessary to induce a bounded composition operator on $H^\infty_\ap(\bC_0)$.

\begin{lemma}\label{lemma:necessity}
    Let $\varphi:\bC_0\to\bC_0$ be an analytic function such that $e^{-\lambda\varphi}\in H^\infty_\ap(\bC_0)$ for some $\lambda>0$. Then $\varphi$ is of the form $\varphi(s)=as+\psi(s)$ for some $a\geq0$ and some analytic function $\psi$ that is almost periodic on $\bC_\kappa$ for all $\kappa>0$.
\end{lemma}

\begin{proof}
    As $e^{-\lambda\varphi}$ is analytic, non-vanishing, and almost periodic on $\bC_\kappa$ for all $\kappa>0$, we can by \cref{thm:bochner} find some $\lambda_0\geq0$ and some analytic function $g$, almost periodic on $\bC_\kappa$ for some $\kappa>0$, such that $g(+\infty)\neq0$ and
    $$f(s)=e^{-\lambda_0s}g(s)$$
    for all $s\in\bC_0$. Set $a=\lambda_0/\lambda$ and
    $$\psi(s)=\varphi(s)-as.$$
    Noting that $g=e^{-\lambda\psi}$, it follows by \cref{lemma:ap-exponent} that $\psi$ is almost periodic on $\bC_\kappa$ for some $\kappa>0$. If $\varphi$ is constant, then it is clearly almost periodic on $\bC_\kappa$ for all $\kappa>0$, and if not, then by \cref{lemma:psi-image-cn} we have that $\psi(\bC_0)\subseteq\bC_0$, so by \cref{thm:strong-bohr} we have that $\psi$ is almost periodic on $\bC_\kappa$ for all $\kappa>0$. Thus $\varphi$ is of the desired form, and the result follows.
\end{proof}

Combining our results we prove now \cref{thm:bounded-composition-characterization}, which we restate for convenience.

\boundedcomposition*

\begin{proof}
    It is clear that \labelcref{item:bounded-composition} and \labelcref{item:composition-is-in-hinftyap} are equivalent, and that \labelcref{item:composition-is-in-hinftyap} implies \labelcref{item:exponentials-are-in-hinftyap}. By adapting the argument in the proof of \cite{bayart_composition_2021}*{Lemma 1}, one shows that \labelcref{item:exponentials-are-in-hinftyap} implies that $\varphi$ is analytic, and in particular we get that \labelcref{item:exponentials-are-in-hinftyap} implies \labelcref{item:exponential-and-analytic}. That \labelcref{item:exponential-and-analytic} implies \labelcref{item:phi-of-form} follows by \cref{lemma:necessity}, and finally, that \labelcref{item:phi-of-form} implies \labelcref{item:composition-is-in-hinftyap} follows from \cref{lemma:sufficiency}.
\end{proof}

\begin{remark}
    It is worth noting that the assumption of analyticity in \labelcref{item:exponential-and-analytic} cannot be dropped. Indeed take any analytic $\varphi$ with $e^{-\lambda\varphi}\in H^\infty_\ap(\bC_0)$ for some $\lambda>0$, and define $\phi(s)=\varphi(s)+2\pi in(s)/\lambda$ for some non-constant integer-valued function $n$. Then $\phi$ is not analytic, but $e^{-\lambda\phi}=e^{-\lambda\varphi}\in H^\infty_\ap(\bC_0)$.
\end{remark}

Before moving on to compactness, we will use \cref{thm:bounded-composition-characterization} to characterize the symbols $\varphi:\bC_0\to\bC_0$ which induce bounded composition operators $C_\varphi:A_\ap(\bC_0)\to A_\ap(\bC_0)$. The analogous characterization was done for $\sA(\bC_0)$ by Contreras, Gómez-Cabello and Rodríguez-Piazza, who showed that the symbols $\varphi:\bC_0\to\bC_0$ inducing bounded composition operators $C_\varphi:\sA(\bC_0)\to\sA(\bC_0)$ are precisely those that induce bounded composition operators on $H^\infty_\ap(\bC_0)$, and in addition are uniformly continuous on the set
$$\{s\in\bC_0:\Re\varphi(s)<r\}$$
for all $r>0$ (see \cite{contreras_composition_2024}*{Theorem 2.13}). We will show that we have the analogous characterization in our setting by adapting their argument. The main tool need is a purely function theoretic lemma, which was baked into the argument of \cite{contreras_composition_2024}*{Theorem 2.13}. We give a different proof.

\begin{lemma}\label{lemma:uniform-continuity}
    Let $f:\bC_0\to\bC_0$ be analytic, and suppose that $e^{-\lambda f}$ is uniformly continuous for some $\lambda>0$. Then $f$ is uniformly continuous on the set
    \begin{equation}\label{eq:sublevel-set}
        \{s\in\bC_0:\Re f(s)<r\}
    \end{equation}
    for all $r>0$.
\end{lemma}

\begin{proof}
    Suppose without loss of generality that $\lambda=1$, and set $g=e^{-f}$. Fix $r>0$ and denote by $\Omega_r$ the set in \cref{eq:sublevel-set}. Let $\varepsilon>0$. Set $R=e^{-r}$ so that $s\in\Omega_r$ if and only if $\abs{g(s)}>R$. As $g$ is uniformly continuous, we can find some $\delta>0$ such that if $s_1,s_2\in\bC_0$ satisfy $\abs{s_1-s_2}<\delta$, then
    $$\abs{g(s_1)-g(s_2)}\leq\frac{R}{2}\min\{1,\varepsilon\}.$$
    Fix any $s_1,s_2\in\Omega_r$ with $\abs{s_1-s_2}<\delta$. Observe first that if $s\in B_\delta(s_1)$, then $g(s)\in B_{R/2}(g(s_1))$. As $B_{R/2}(g(s_1))$ is bounded away from zero, we can let $\log$ denote the principal branch of the logarithm defined on $B_{R/2}(g(s_1))$. Observe that then $f+\log g$ is constant on $B_\delta(s_1)$, by which we can estimate
    $$\abs{f(s_1)-f(s_2)}=\abs{\log g(s_1)-\log g(s_2)}=\abs{\int_{g(s_2)}^{g(s_1)}\frac{\dd\zeta}{\zeta}}\leq\frac{2}{R}\abs{g(s_1)-g(s_2)}\leq\varepsilon.$$
    The result follows.
\end{proof}

With this we now prove our characterization theorem on the bounded composition operators on $A_\ap(\bC_0)$.

\begin{theorem}\label{thm:bounded-composition-characterization-algebra}
    Let $\varphi:\bC_0\to\bC_0$ be a function. The following are equivalent:
    \begin{enumerate}[(i)]
        \item\label{item:bounded-composition-algebra} $\varphi$ induces a bounded composition operator $C_\varphi:A_\ap(\bC_0)\to A_\ap(\bC_0)$.
        \item\label{item:composition-is-in-aap} $f\circ\varphi\in A_\ap(\bC_0)$ for all $f\in A_\ap(\bC_0)$.
        \item\label{item:exponentials-are-in-aap} $e^{-\lambda\varphi}\in A_\ap(\bC_0)$ for all $\lambda>0$.
        \item\label{item:exponential-and-analytic-algebra} $e^{-\lambda\varphi}\in A_\ap(\bC_0)$ for some $\lambda>0$, and $\varphi$ is analytic.
        \item\label{item:phi-of-form-algebra} $\varphi$ is of the form $\varphi(s)=as+\psi(s)$ for some $a\geq0$ and some analytic function $\psi$ that is almost periodic on $\bC_\kappa$ for all $\kappa>0$, and $\varphi$ is uniformly continuous on the set
        $$\{s\in\bC_0:\Re\varphi(s)<r\}$$
        for all $r>0$.
    \end{enumerate}
\end{theorem}

\begin{proof}
    That \labelcref{item:bounded-composition-algebra} and \labelcref{item:composition-is-in-aap} are equivalent is clear, and so is the fact that \labelcref{item:composition-is-in-aap} implies \labelcref{item:exponentials-are-in-aap}. That \labelcref{item:exponentials-are-in-aap} implies \labelcref{item:exponential-and-analytic-algebra} follows by applying the corresponding implication in \cref{thm:bounded-composition-characterization}. That \labelcref{item:exponential-and-analytic-algebra} implies \labelcref{item:phi-of-form-algebra} follows by applying the corresponding implication in \cref{thm:bounded-composition-characterization} together with \cref{lemma:uniform-continuity}. We show that \labelcref{item:phi-of-form-algebra} implies \labelcref{item:composition-is-in-aap}. Suppose \labelcref{item:phi-of-form-algebra} holds, and let $f\in A_\ap(\bC_0)$. By \cref{thm:bounded-composition-characterization} we have that $f\circ\varphi\in H^\infty_\ap(\bC_0)$, and so we only need to show that $f\circ\varphi$ is uniformly continuous. Let $\varepsilon>0$. As $f$ is uniformly continuous, we can find some $\rho>0$ such that if $s_1,s_2\in\bC_0$ satisfy $\abs{s_1-s_2}\leq\rho$, then
    \begin{equation}\label{eq:f-uniform}
        \abs{f(s_1)-f(s_2)}\leq\varepsilon.
    \end{equation}
    Next, use \cref{lemma:limit-at-infinity} to find a $\kappa>0$ such that
    \begin{equation}\label{eq:f-infinity}
        \sup_{s\in\overline\bC_\kappa}\abs{f(s)-f(+\infty)}\leq\varepsilon.
    \end{equation}
    For $r>0$, write $\Omega_r=\{s\in\bC_0:\Re\varphi(s)<r\}$. As $\varphi$ is uniformly continuous on $\Omega_{\kappa/2}$, we can find some $\delta>0$ such that if $s_1,s_2\in\Omega_{\kappa/2}$ satisfy $\abs{s_1-s_2}\leq\delta$, then
    \begin{equation}\label{eq:phi-uniform}
        \abs{\varphi(s_1)-\varphi(s_2)}\leq\rho.
    \end{equation}
    Take now any $s_1,s_2\in\bC_0$ with $\abs{s_1-s_2}\leq\delta$. If $s_1,s_2\in\Omega_\kappa$, then \cref{eq:f-uniform} and \cref{eq:phi-uniform} imply that
    $$\abs{f(\psi(s_1))-f(\psi(s_2))}\leq\varepsilon.$$
    If $s_1,s_2\in\bC_0\setminus\Omega_{\kappa/2}$, then $\varphi(s_1),\varphi(s_2)\in\overline\bC_\kappa$, so \cref{eq:f-infinity} implies that
    $$\abs{f(\varphi(s_1))-f(\varphi(s_2))}\leq\abs{f(\varphi(s_1))-f(+\infty)}+\abs{f(\varphi(s_2))-f(+\infty)}\leq2\varepsilon.$$
    Finally, suppose $s_1\in\Omega_{\kappa/2}$ and $s_2\in\bC_0\setminus\Omega_\kappa$. By an application of the intermediate value theorem to the function $\lambda\mapsto\Re\varphi(\lambda s_1+(1-\lambda)s_2)$ we can find some $\lambda\in[0,1]$ such that $s'=\lambda s_1+(1-\lambda)s_2\in\Omega_\kappa\setminus\Omega_{\kappa/2}$. As $\abs{s_1-s'}\leq\abs{s_1-s_2}\leq\delta$, we can then use the previous two cases to estimate
    $$\abs{f(\varphi(s_1))-f(\varphi(s_2))}\leq\abs{f(\varphi(s_1))-f(\varphi(s'))}+\abs{f(\varphi(s'))-f(\varphi(s_2))}\leq3\varepsilon.$$
    We have thus shown that any $s_1,s_2\in\bC_0$ with $\abs{s_1-s_2}\leq\delta$ satisfy
    $$\abs{f(\varphi(s_1))-f(\varphi(s_2))}\leq3\varepsilon,$$
    and as $\varepsilon>0$ was arbitrary, this shows that $f\circ\varphi$ is uniformly continuous. Thus $f\circ\varphi\in A_\ap(\bC_0)$, and the result follows.
\end{proof}

With boundedness taken care of, we now move on to compactness. We start by characterizing the compact composition operators on $H^\infty_\ap(\bC_0)$; that is, we prove \cref{thm:compact-composition}, which we restate for convenience.

\compactcomposition*

\begin{proof}
    Suppose first that \labelcref{item:compactly-contained} holds. Let $\{f_n\}_{n\in\bZ^+}$ be a uniformly bounded sequence in $H^\infty_\ap(\bC_0)$. By the classical Montel theorem (e.g. \cite{rudin_real_1987}*{Theorem 14.6}), we can find an $f\in H^\infty(\bC_0)$ and a subsequence $\{f_{n_k}\}_{k\in\bZ^+}$ that converges uniformly to $f$ on all compact subsets of $\bC_0$.  As $\varphi(\bC_0)$ is compactly contained in $\bC_0$, this then implies that $\{f_{n_k}\circ\varphi\}_{k\in\bZ^+}$ converges uniformly to $f\circ\varphi\in H^\infty_\ap(\bC_0)$. This shows \labelcref{item:compact}.

    For the reverse implication, suppose \labelcref{item:compact} holds. Consider the two sequences $\{f_n\}_{n\in\bZ^+}$ and $\{g_n\}_{n\in\bZ^+}$ in $H^\infty_\ap(\bC_0)$ defined by
    $$f_n(s)=e^{-s/n}-1,\quad g_n(s)=e^{-ns}.$$
    These two sequences are clearly uniformly bounded and both converge pointwise to zero. As $C_\varphi$ is compact, the sequences $\{f_n\circ\varphi\}_{n\in\bZ^+}$ and $\{g_n\circ\varphi\}_{n\in\bZ^+}$ must then each have a subsequence converging uniformly to the zero function. We show first that $\Im\varphi$ is bounded. Suppose it was not. Then, for each $n\in\bZ^+$, we can find some $s_n\in\bC_0$ such that
    $$\abs{\Im\varphi(s_n)}=\pi n+2\pi nk_n$$
    for some $k_n\in\bZ^+$. Compute then
    $$\abs{f_n(s)}^2=e^{-2\Re s/n}+1-2e^{-\Re s/n}\cos\left(\frac{\abs{\Im s}}{n}\right).$$
    From this we see that
    $$\norm{f_n\circ\varphi}_\infty^2\geq\abs{f_n(\varphi(s_n))}^2=e^{-2\Re s_n/n}+1+2e^{-\Re s_n/n}\geq1$$
    for all $n\in\bZ^+$, contradicting that $\{f_n\circ\varphi\}_{n\in\bZ^+}$ has a uniformly convergent subsequence. We show next that $\Re\varphi$ is bounded from above. Suppose it was not, and let $\{s_k\}_{k\in\bZ^+}$ be a sequence in $\bC_0$ with $\Re\varphi(s_k)\to\infty$ as $k\to\infty$. Then
    $$\norm{f_n\circ\varphi}_\infty\geq\limsup_{k\to\infty}\lvert e^{-\varphi(s_k)/n}-1\rvert\geq\lim_{k\to\infty}(1-e^{-\Re\varphi(s_k)/n})=1$$
    for all $n\in\bZ^+$, again contradicting that $\{f_n\circ\varphi\}_{n\in\bZ^+}$ has a uniformly convergent subsequence. We finally show that $\Re \varphi$ is bounded away from zero. Suppose it was not, and let $\{s_k\}_{k\in\bZ^+}$ be a sequence in $\bC_0$ such that $\Re\varphi(s_k)\to0$ as $k\to\infty$. Then
    $$\norm{g_n\circ\varphi}_\infty\geq\lim_{k\to\infty}\abs{g_n(\varphi(s_k))}=\lim_{k\to\infty}e^{-n\Re\varphi(s_k)}=1$$
    for all $n\in\bZ^+$, contradicting that $\{g_n\circ\varphi\}_{n\in\bZ^+}$ has a uniformly convergent subsequence. We have thus shown that $\varphi(\bC_0)$ lies in a bounded subset of $\bC_0$ that is bounded away from $\partial\bC_0$, which is equivalent to $\varphi(\bC_0)$ being compactly contained in $\bC_0$, which shows \labelcref{item:compactly-contained}.
\end{proof}

The above theorem is in contrast to the situation in $\sH^\infty$, where compactness of the composition operator $C_\varphi:\sH^\infty\to\sH^\infty$ is equivalent to the existence of some $\kappa>0$ such that $\varphi(\bC_0)\subseteq\bC_\kappa$ (see \cite{bayart_hardy_2002}*{Theorem 18}). We show now that in subspaces of $H^\infty_\ap(\bC_0)$ where uniform boundedness implies joint almost periodicity on some half-plane---such as $\sH^\infty$, $\sD_\ext^\infty(\lambda)$ for a frequency $\lambda$ with $L(\lambda)<\infty$, and $H^\infty_\lambda(\bC_0)$ for an arbitrary frequency $\lambda$---the assumption that $\varphi(\bC_0)\subseteq\bC_\kappa$ for some $\kappa>0$ is sufficient for $C_\varphi$ to be compact. This is a direct consequence of the strong Montel theorem holding in these spaces, and the proof uses \cref{thm:strong-montel} analogously to how Bayart used his strong Montel theorem in the proof of \cite{bayart_hardy_2002}*{Theorem 18}.

\begin{theorem}\label{thm:compact-composition-w-sufficiency}
    Let $W$ be a subspace of $H^\infty_\ap(\bC_0)$ with the property that any uniformly bounded family in $W$ is jointly almost periodic on some half-plane $\bC_\alpha$ with $\alpha>0$, and let $\varphi:\bC_0\to\bC_0$ induce a bounded composition operator $C_\varphi:W\to H^\infty_\ap(\bC_0)$. If $\varphi(\bC_0)\subseteq\bC_\kappa$ for some $\kappa>0$, then $C_\varphi$ is compact.
\end{theorem}

\begin{proof}
    Let $\{f_n\}_{n\in\bZ^+}$ be a uniformly bounded sequence in $W$. By assumption it is then jointly almost periodic on some half-plane $\bC_\alpha$ with $\alpha>0$, so by \cref{thm:strong-montel} we can find some $f\in H^\infty_\ap(\bC_0)$ and some subsequence $\{f_{n_k}\}_{k\in\bZ^+}$ that converges uniformly to $f$ on all half-planes $\bC_\beta$ with $\beta>0$. In particular, as $\varphi(\bC_0)\subseteq\bC_\kappa$, we have that $\{f_{n_k}\circ\varphi\}_{k\in\bZ^+}$ converges uniformly to $f\circ\varphi\in H^\infty_\ap(\bC_0)$. The result follows.
\end{proof}

It is clear that the condition that $\varphi(\bC_0)\subseteq\bC_\kappa$ for some $\kappa>0$ is not necessary for compactness in general. For example, if $W$ is a finite-dimensional subspace of $H^\infty_\ap(\bC_0)$, then any operator on it (so in particular any composition operator $C_\varphi:W\to H^\infty_\ap(\bC_0)$) is compact. If we require $W$ to be large enough in the sense of containing some sequence of functions of the form $s\mapsto e^{-\lambda_ns}$ with $\lambda_n\to\infty$, then the condition is necessary as well. The proof of this is identical to how we showed that $\Re\varphi$ has to be bounded away from zero in the proof of \cref{thm:compact-composition}, with this sequence taking the role of $\{g_n\}_{n\in\bZ^+}$ in that proof, and so we leave the details to the reader.

\begin{theorem}\label{thm:compact-composition-w-necessity}
    Let $W$ be a subspace of $H^\infty_\ap(\bC_0)$ and let $\varphi:\bC_0\to\bC_0$ induce a bounded composition operator $C_\varphi:W\to H^\infty_\ap(\bC_0)$. Suppose that there exists a sequence $\{\lambda_n\}_{n\in\bZ^+}$ of non-negative real numbers with limit $\infty$ such that the functions $s\mapsto e^{-\lambda_ns}$ all belong to $W$. If $C_\varphi$ is compact, then $\varphi(\bC_0)\subseteq\bC_\kappa$ for some $\kappa>0$.
\end{theorem}

With this we obtain the following corollary.

\begin{corollary}
    Let $\lambda$ be a frequency, let $W$ be either $\sD_\ext^\infty(\lambda)$ with $L(\lambda)<\infty$ or $H^\infty_\lambda(\bC_0)$, and let $\varphi:\bC_0\to\bC_0$ induce a bounded composition operator $C_\varphi:W\to H^\infty_\ap(\bC_0)$. The following are equivalent:
    \begin{enumerate}[(i)]
        \item $C_\varphi$ is compact.
        \item $\varphi(\bC_0)\subseteq\bC_\kappa$ for some $\kappa>0$.
    \end{enumerate}
\end{corollary}

\begin{proof}
    By definition $W$ satisfies the assumptions of \cref{thm:compact-composition-w-necessity}, from which the first implication follows. By \cref{thm:l-lambda-montel}, \cref{thm:defant-montel} and \cref{thm:strong-montel}, we have that $W$ satisfies the assumptions of \cref{thm:compact-composition-w-sufficiency}, from which the other implication follows.
\end{proof}

In particular, as $\sH^\infty=\sD_\ext^\infty(\{\log n\}_{n\in\bZ^+})$, we obtain the characterization of the compact composition operators on $\sH^\infty$ of Bayart \cite{bayart_hardy_2002}*{Theorem 18}.

\begin{example}
    Consider the map $\varphi_a(s)=a+s$ for $s\in\bC_0$ and $a\geq0$. As $\varphi_a(\bC_0)=\bC_a$, we see that the operator $C_\varphi$ is never compact on $H^\infty_\ap(\bC_0)$, but is compact on $H^\infty_\lambda(\bC_0)$ for an arbitrary frequency $\lambda$ if and only if $a>0$.
\end{example}

\bibliography{references}

\end{document}